\documentclass[a4paper]{article}

\newcommand{\email}[1]{email: #1} 
\newcommand{\institute}[1]{\date{#1}}
\newenvironment{acknowledgement}{\paragraph{Acknowledgements}\small}{}

\usepackage{makeidx}         
\usepackage{graphicx}        
\usepackage[bottom]{footmisc}

\usepackage{latexsym}
\usepackage{amsmath}
\usepackage{amsfonts}
\usepackage{amssymb}
\usepackage{amsthm}
\usepackage{bm}

\usepackage{url}
\usepackage{algorithm}
\usepackage{algorithmic}
\usepackage[misc,geometry]{ifsym}

\theoremstyle{plain}
  
  \newtheorem{proposition}{Proposition}
  \newtheorem{lemma}{Lemma}
  \newtheorem{corollary}{Corollary}
  \newtheorem{conjecture}{Conjecture}
  
\theoremstyle{definition}

\theoremstyle{remark}

\newcommand{\bsa}{{\boldsymbol{a}}}

\newcommand{\bsh}{{\boldsymbol{h}}}

\newcommand{\bsw}{{\boldsymbol{w}}}
\newcommand{\bsx}{{\boldsymbol{x}}}
\newcommand{\bsy}{{\boldsymbol{y}}}
\newcommand{\bsz}{{\boldsymbol{z}}}

\newcommand{\bszero}{{\boldsymbol{0}}} 
\newcommand{\bsone}{{\boldsymbol{1}}}  

\newcommand{\bsgamma}{{\boldsymbol{\gamma}}}

\newcommand{\rd}{{\mathrm{d}}}

\newcommand{\R}{{\mathbb{R}}} 
\newcommand{\Z}{{\mathbb{Z}}} 

\DeclareSymbolFont{bbold}{U}{bbold}{m}{n}
\DeclareSymbolFontAlphabet{\mathbbold}{bbold}
\newcommand{\ind}{{\mathbbold{1}}}

\newcommand{\calH}{{\mathcal{H}}}

\newcommand{\setu}{{\mathfrak{u}}}
\newcommand{\setv}{{\mathfrak{v}}}

\usepackage{a4wide}

\DeclareMathOperator{\wce}{wce}
\newcommand{\sob}{\mathrm{usob}}
\newcommand{\kor}{\mathrm{kor}}
\newcommand{\lin}{\mathrm{lin}}
\newcommand{\m}{\mathrm{vm}}
\newcommand{\imagunit}{\mathrm{i}}
\newcommand{\twopii}{2\pi\imagunit\,}

\usepackage[colorlinks=true,linkcolor=blue,citecolor=blue]{hyperref}

\begin{document}

\title{The Analysis of Vertex Modified Lattice Rules\\ in a Non-Periodic Sobolev Space}

\author{Dirk Nuyens \and Ronald Cools}

\date{Department of Computer Science, KU Leuven, Belgium \\
 \email{dirk.nuyens@cs.kuleuven.be; ronald.cools@cs.kuleuven.be}}

\maketitle


\paragraph{{\sl Dedicated to Ian H.~Sloan's beautiful contributions to the existence and construction of lattice rules, on the occasion of his 80th birthday.}}

\abstract{
  In a series of papers, in 1993, 1994 \& 1996 (see \cite{NS93,NS94,NS96}), Ian Sloan together with Harald Niederreiter introduced a modification of lattice rules for non-periodic functions, called ``vertex modified lattice rules'', and a particular breed called ``optimal vertex modified lattice rules''.
  These are like standard lattice rules but they distribute the point at the origin to all corners of the unit cube, either by equally distributing the weight and so obtaining a multi-variate variant of the trapezoidal rule, or by choosing weights such that multilinear functions are integrated exactly.
  In the 1994 paper, Niederreiter \& Sloan concentrate explicitly on Fibonacci lattice rules, which are a particular good choice of 2-dimensional lattice rules.
  Error bounds in this series of papers were given related to the star discrepancy.
\\
  In this paper we pose the problem in terms of the so-called unanchored Sobolev space, which is a reproducing kernel Hilbert space often studied nowadays in which functions have $L_2$-integrable mixed first derivatives.
  It is known constructively that randomly shifted lattice rules, as well as deterministic tent-transformed lattice rules and deterministic fully symmetrized lattice rules can achieve close to $O(N^{-1})$ convergence in this space, see Sloan, Kuo \& Joe (2002, see~\cite{SKJ2002}) and Dick, Nuyens \& Pillichshammer (2014, see~\cite{DNP2014}) respectively, where possible $\log(N)^s$ terms are taken care of by weighted function spaces.
\\
  We derive a break down of the worst-case error of vertex modified lattice rules in the unanchored Sobolev space in terms of the worst-case error in a Korobov space, a multilinear space and some additional ``mixture term''. 
  For the 1-dimensional case this worst-case error is obvious and gives an explicit expression for the trapezoidal rule.
  In the 2-dimensional case this mixture term also takes on an explicit form for which we derive upper and lower bounds.
  For this case we prove that there exist lattice rules with a nice worst-case error bound with the additional mixture term of the form $N^{-1} \log^2(N)$.
}

\section{Introduction}

We study the numerical approximation of an $s$-dimensional integral over the unit cube
\begin{align*}
  I(f)
  &:=
  \int_{[0,1]^s} f(\bsx) \,\rd{\bsx}
  .
\end{align*}
A (rank-1) \emph{lattice rule} with $N$ points in $s$ dimensions is an equal weight cubature rule
\begin{align}\label{eq:latticerule}
  Q(f; \bsz, N)
  &:=
  \frac1N \sum_{k=0}^{N-1} f\left( \left\{\frac{\bsz k}{N}\right\} \right)
  ,
\end{align}
where $\bsz \in \Z^s$ is the \emph{generating vector} of which the components are most often chosen to be relatively prime to $N$, and the curly braces $\{ \cdot \}$ mean to take the fractional part componentwise.
Clearly, as this is an equal weight rule, the constant function is integrated exactly.
The classical theory, see \cite{Nie92,SJ94}, is mostly concerned with periodic functions and then uses the fact that $f$ can be expressed in an absolutely converging Fourier series to study the error.
See also \cite{Nuy2014} for a recent overview of this ``spectral'' error analysis and its application to lattice rules.
In this paper we only consider real-valued integrand functions.

In a series of papers \cite{NS93,NS94,NS96} Niederreiter and Sloan introduced vertex modified lattice rules, and, more general, vertex modified quasi-Monte Carlo rules, to also cope with non-periodic functions.
In this paper we revisit these vertex modified lattice rules using the technology of reproducing kernel Hilbert spaces, more precisely the unanchored Sobolev space of smoothness~1.
The inner product for the one-dimensional unanchored Sobolev space is defined by
\begin{align}\label{eq:sobip}
  \langle f, g \rangle_{\sob1,1,\gamma_1}
  &:=
  \int_{0}^{1} f(x) \,\rd{x} \, \int_{0}^{1} g(x) \,\rd{x}
  +
  \frac1{\gamma_1} \int_{0}^{1} f'(x) \,g'(x) \,\rd{x}
  ,
\end{align}
where, more generally, $\gamma_j$ is a ``product weight'' associated with dimension~$j$, which is used to model the importance of different dimensions, see, e.g., \cite{SW98}.
In the multivariate case we take the tensor product such that the norm is defined by
\begin{align}\notag
  \|f\|_{\sob1,s,\bsgamma}^2
  &:=
  \sum_{\setu \subseteq \{1:s\}} \gamma_{\setu}^{-1}
  \int_{[0,1]^{|\setu|}}
  \left( \int_{[0,1]^{s-|\setu|}} \frac{\partial^{|\setu|}}{\partial\bsx_{\setu}} f(\bsx) \,\rd{\bsx_{-\setu}} \right)^2
  \rd{\bsx_{\setu}}
  \\\label{eq:sobnorm}
  &\hphantom{:}=
  \sum_{\setu \subseteq \{1:s\}} \gamma_{\setu}^{-1}
  \left\| \int_{[0,1]^{s-|\setu|}} \frac{\partial^{|\setu|}}{\partial\bsx_{\setu}} f(\bsx) \,\rd{\bsx_{-\setu}} \right\|^2_{L_2}
  ,
\end{align}
with $\gamma_{\setu} = \prod_{j\in\setu} \gamma_j$.
We use the short hand notation $\{1:s\} = \{1,\ldots,s\}$ and thus in~\eqref{eq:sobnorm}
$\setu$ ranges over all subsets of $\{1,\ldots,s\}$, and $-\setu$ is the complement with respect to the full set, $-\setu = \{1:s\} \setminus \setu$.
Note that~\eqref{eq:sobnorm} is a sum of $L_2$-norms of mixed first derivatives for all variables in $\setu$ where all other variables are averaged out.

\section{Vertex Modified Lattice Rules}

The \emph{vertex modified lattice rule} proposed in \cite{NS93} is given by
\begin{align}\label{eq:Qvm}
  Q^{\m}(f; \bsz, N, \bsw)
  &=
  \sum_{\bsa \in \{0,1\}^s} w(\bsa) f(\bsa)
  +
  \frac1N \sum_{k=1}^{N-1} f\left(\left\{\frac{\bsz k}{N}\right\}\right)
  ,
\end{align}
with well chosen vertex weights $w(\bsa)$ such that the constant function is still integrated exactly.
It is assumed that $\gcd(z_j,N) = 1$, for all $j=1,\ldots,s$, such that only the lattice point for $k=0$ is on the edge of the domain $[0,1]^s$, and this is why the second sum only ranges over $k=1,\ldots,N-1$, i.e., the interior points.
We note that typically $N$ equals the number of function evaluations. This is not true anymore for vertex modified lattice rules.
We define $M$ to be the total number of function evaluations, and this is given by
\begin{align}\label{eq:M}
  M 
  &= 
  2^s + N - 1
  .
\end{align}
The $2^s$ term makes us focus on the low-dimensional cases only, and we derive explicit results for $s=2$ later.
The vertex modified rule can then be represented as a standard cubature rule of the form
\begin{align}
  Q(f; \{(w_k,\bsx_k)\}_{k=1}^M) 
  = 
  Q(f) 
  &= 
  \sum_{k=1}^M w_k \, f(\bsx_k)
  ,
\end{align}
with appropriate choices for the pairs $(w_k, \bsx_k)$.
For the vertex modified rules we only need to specify the weights at the vertices of the unit cube, all other remain unchanged from the standard lattice rule and are $1/N$.

Two particular choices for the weights $w(\bsa)$ have been proposed \cite{NS93,NS94,NS96}.
The first one has constant weights $w(\bsa) \equiv 1 / (2^s N)$ which mimics the trapezoidal rule in each one-dimensional projection:
\begin{align*}
  T(f; \bsz, N)
  &:=
  Q^{\m}(f; \bsz, N, \frac1{2^s N})
  =
  \frac1{2^s N} \sum_{\bsa \in \{0,1\}^s} f(\bsa)
  +
  \frac1N \sum_{k=1}^{N-1} f\left(\left\{\frac{\bsz k}{N}\right\}\right)
  .
\end{align*}
A second particular choice of weights $w^*(\bsa)$ leads to the so-called \emph{optimal vertex modified lattice rule} \cite{NS93}:
\begin{align*}
  Q^{*}(f; \bsz, N)
  &:=
  Q^{\m}(f; \bsz, N, \bsw^*)
  =
  \sum_{\bsa \in \{0,1\}^s} w^*(\bsa) f(\bsa)
  +
  \frac1N \sum_{k=1}^{N-1} f\left(\left\{\frac{\bsz k}{N}\right\}\right)
  .
\end{align*}
This rule integrates all multilinear polynomials exactly, i.e.,
\begin{align*}
  Q^{*}(f; \bsz, N)
  =
  Q^{\m}(f; \bsz, N, \bsw^*)
  &=
  I(f)
  &\text{for all}&&
  f(\bsx)
  &=
  \prod_{j=1}^s x_j^{k_j}
  \quad \text{with }
  k_j \in \{0,1\}
  .
\end{align*}

There is no need to solve a linear system of equations to find the weights $w^*(\bsa)$.
The following result from \cite{NS93} shows they can be determined explicitly.
\begin{proposition}
  For every $\bsa \in \{0,1\}^s$ define $\setu$ to be the support of $\bsa$, i.e., $\setu = \setu(\bsa) = \{ 1 \le j \le s : a_j \ne 0 \}$.
  Then the weight $w^*(\bsa)$ is given by
  \begin{align*}
    w^*(\bsa)
    =
    w^{*}_{\setu}
    &=
    \frac1{2^s} - \frac1N \sum_{k=1}^{N-1} \ell_{\setu}\left(\left\{\frac{\bsz k}{N}\right\}\right)
    &\text{where}&&
    \ell_{\setu}(\bsx)
    &:=
    \prod_{j \in \setu} x_j \prod_{j \in \{1:s\} \setminus \setu} (1-x_j)
    .
  \end{align*}
\end{proposition}
\begin{proof}
The idea is to use a kind of a Lagrange basis which is $0$ in all vertex points $\bsa \in \{0,1\}^s$ except in one.
For this purpose, consider the basis, for $\setu \subseteq \{1:s\}$,
\begin{align*}
  \ell_{\setu}(\bsx)
  &=
  \prod_{j \in \setu} x_j \prod_{j \in \{1:s\} \setminus \setu} (1-x_j)
\end{align*}
such that $\ell_{\setu}(\bsa) = \ind_{\setu(\bsa) = \setu}$.
Demanding that $Q(\ell_{\setu}) = I(\ell_{\setu})$ for some $\setu \subseteq \{1:s\}$, gives
\begin{align*}
  \sum_{\bsa \in \{0,1\}^s} w^*(\bsa) \, \ell_{\setu}(\bsa) 
  + \frac1N \sum_{k=1}^{N-1} \ell_{\setu}\left(\left\{\frac{\bsz k}{N}\right\}\right)
  &=
  \int_{[0,1]^s} \ell_{\setu}(\bsx) \,\rd{\bsx}
\end{align*}
from where the result follows.
\end{proof}

\section{Reproducing Kernel Hilbert Spaces}

In this section we collect some well known results.
For more details the reader is referred to, e.g., \cite{Hic98a,NW2008,DKS2013,DNP2014,Nuy2014}.

The reproducing kernel $K : [0,1] \times [0,1] \to \R$ of a one-dimensional reproducing kernel Hilbert space $\calH(K)$ is a symmetric, positive definite function which has the reproducing property
\begin{align*}
  f(y) = \langle f, K(\cdot,y) \rangle_{K}
  \quad
  \text{for all } f \in \calH(K) \text{ and } y \in [0,1].
\end{align*}
The induced norm in the space will be denoted by $\|f\|_{K} = \sqrt{\langle f, f \rangle_{K}}$.
If the space has a countable basis $\{\varphi_{h}\}_{h}$ which is orthonormal with respect to the inner product of the space, then, by virtue of Mercer's theorem, the kernel is given by
\begin{align*}
  K(x, y)
  &=
  \sum_{h} \varphi_{h}(x) \, \overline{\varphi_{h}(y)}
  .
\end{align*}
For the multivariate case we consider the tensor product space and the kernel is then given by
\begin{align*}
  K_{s}(\bsx, \bsy)
  &=
  \prod_{j=1}^s K(x_j,y_j)
  .
\end{align*}

We define the \emph{worst-case error} of integration using a cubature rule $Q$ to be
\begin{align*}
  \wce(Q; K)
  &:=
  \sup_{\substack{f \in \calH(K) \\ \|f\|_{K} \le 1}} | Q(f) - I(f) |
  .
\end{align*}
For a general cubature formula $Q(f) = \sum_{k=1}^{M} w_{k} \, f(\bsx_k)$ the squared worst-case error can be written as, see, e.g., \cite{Hic98a},
\begin{align}\label{eq:wce}
  \wce(Q; K)^{2}
  &=
  \int_{[0,1]^{2s}} K(\bsx,\bsy) \,\rd{\bsx}\rd{\bsy}
  -
  2 \sum_{k=1}^{M} w_{k} \int_{[0,1]^{s}} K(\bsx_k, \bsy) \,\rd{\bsy}
  +
  \sum_{k,\ell=1}^{M} w_{k}w_{\ell} \, K(\bsx_k, \bsx_{\ell})
  .
\end{align}
For all kernels in the remainder of the text we have that $\int_0^1\int_0^1 K(x,y) \,\rd{x}\rd{y} = 1$ and $\int_0^1 K(x,y) \,\rd{y} = 1$ for all $x \in [0,1]$ and this also holds for the multivariate kernel due to the product structure.

\subsection{The Korobov Space}\label{sec:Korobov-space}

A well known example is the Korobov space which consists of periodic functions which can be expanded in an absolutely converging Fourier series.
We refer the reader to the general references in the beginning of this section for further information on the Korobov space.
Denote the Fourier coefficients by
\begin{align*}
  \hat{f}(\bsh)
  &:=
  \int_{[0,1]^s} f(\bsx) \, \exp(-\twopii \bsh \cdot \bsx) \,\rd{\bsx}
  ,
  \qquad
  \bsh \in \Z^s
  .
\end{align*}
In the one-dimensional case, if we assume an algebraic decay of $h^{-\alpha}$, $\alpha > 1/2$, by means of
\begin{align*}
  \|f\|^{2}_{\kor \alpha,1,\gamma_1}
  &:=
  |\hat{f}(0)|^2 +
  \sum_{0 \ne h \in \Z} |\hat{f}(h)|^2 \, \gamma_1^{-1} |h|^{2\alpha}
  < 
  \infty
  ,
\end{align*}
then the reproducing kernel is given by
\begin{align*}
  K^{\kor \alpha}_{1,\gamma_1}(x,y)
  &:=
  1 + \gamma_1 \sum_{0 \ne h \in \Z} \frac{\exp(\twopii h (x-y))}{|h|^{2\alpha}}
  .
\end{align*}
We now specifically concentrate on the case $\alpha = 1$ as this will be of use throughout the paper.
For $\alpha = 1$ the reproducing kernel for the $s$-variate case can be written as
\begin{align*}
  K^{\kor 1}_{s,\bsgamma}(\bsx, \bsy)
  &=
  \prod_{j=1}^{s} \left( 1 + 2 \pi^2 \gamma_j B_2(\{x_j-y_j\}) \right)
  =
  \sum_{\setu \subseteq \{1:s\}} \prod_{j\in\setu} 2 \pi^2 \gamma_j B_2(\{x_j-y_j\})
  ,
\end{align*}
where $B_2(t) = t^{2} - t + \tfrac16 = \frac1{2\pi^2} \sum_{0\ne h \in \Z} \frac{\exp(\twopii h t)}{h^2}$, for $0 \le t \le 1$, is the 2nd degree Bernoulli polynomial and $\bsgamma = \{ \gamma_j \}_{j=1}^s$ is a set of product weights which are normally used to model dimension importance.
Here we will not make use of the weights $\bsgamma$, except for scaling, such that the worst-case error of one space shows up in the worst-case error expression of another space.

For a general cubature rule $Q(f) = \sum_{k=1}^{M} w_k \, f(\bsx_k)$, with $\sum_{k=1}^M w_k = 1$, using~\eqref{eq:wce} one obtains
\begin{align}\label{eq:wceKor}
  \wce(Q; K^{\kor 1}_{s,\bsgamma})^{2}
  &=
  \sum_{k,\ell=1}^{M} w_{k}w_{\ell}
  \sum_{\emptyset \ne \setu \subseteq \{1:s\}} \prod_{j\in\setu} 2 \pi^2 \gamma_j B_2(\{x_{k,j}-x_{\ell,j}\})
  .
\end{align}
In case $Q(f) = Q(f; \bsz, N)$ is a lattice rule then the difference of two points is also a point of the point set and therefore the squared worst-case error formula simplifies to
\begin{align*}
  \wce(Q(\cdot;\bsz,N); K^{\kor 1}_{s,\bsgamma})^{2}
  &=
  \frac1N \sum_{k=0}^{N-1}
  \sum_{\emptyset \ne \setu \subseteq \{1:s\}} \prod_{j\in\setu} 2 \pi^2 \gamma_j B_2(x_{k,j})
  .
\end{align*}
We remark that, apart from the higher cost, using a vertex modified lattice rule in the Korobov space makes no difference to the worst-case error,
\begin{align}\label{eq:wceQmKor}
  \wce(Q^{\m}(\cdot;\bsz,N,\bsw); K^{\kor \alpha}_{s,\bsgamma})
  &=
  \wce(Q(\cdot;\bsz,N); K^{\kor \alpha}_{s,\bsgamma})
  ,
\end{align}
since $K^{\kor \alpha}_{s,\bsgamma}(\bsa, \bszero) = K^{\kor \alpha}_{s,\bsgamma}(\bszero, \bszero)$ for all $\bsa \in \{0,1\}^{s}$ and the weights $w(\bsa)$ are such that they sum to $1/N$ due to the constraint of integrating the constant function exactly.

\subsection{The Space of Multilinear Functions}

Define the following multilinear functions, for $\setu \subseteq \{1:s\}$,
\begin{align*}
  g_{\setu}(\bsx)
  &:=
  \prod_{j \in \setu} \sqrt{12} \, (x_j - \tfrac12)
  =
  \prod_{j \in \setu} \sqrt{12} \, B_1(x_j)
  ,
\end{align*}
so $g_{\emptyset}(\bsx) = 1$, $g_{\{1\}}(\bsx) = \sqrt{12} \, (x_1 - \tfrac12)$ and so on, where $B_1(t) = t - \tfrac12$ is the 1st degree Bernoulli polynomial.
These functions form an orthonormal basis $\{ g_{\setu} \}_{\setu \subseteq \{1:s\}}$ with respect to the standard $L_2$ inner product
and we can thus construct a reproducing kernel for this finite dimensional space:
\begin{align*}
  K^{\lin}_{s,\bsgamma}(\bsx, \bsy)
  &:=
  \sum_{\setu \subseteq \{1:s\}} \gamma_{\setu} \, g_{\setu}(\bsx) \, g_{\setu}(\bsy)
  =
  1
  +
  \sum_{\emptyset \ne \subseteq \{1:s\}}
    \prod_{j \in \setu} 12 \, \gamma_j \, B_1(x_j) \, B_1(y_j)
  ,
\end{align*}
where we introduced standard product weights.
The worst-case error for a general cubature rule $Q(f) = \sum_{k=1}^{M} w_k \, f(\bsx_k)$, for which $\sum_{k=1}^M w_k = 1$, is given by
\begin{align}\label{eq:wcelin}
  \wce(Q; K^{\lin}_{s,\bsgamma})^{2}
  &=
  \sum_{k,\ell=1}^{M} w_k w_\ell \sum_{\emptyset \ne \setu \subseteq \{1:s\}} \prod_{j\in\setu} 12 \, \gamma_j \, (x_{k,j} - \tfrac12) \, (x_{\ell,j} - \tfrac12)
  .
\end{align}

We remark that this space is not such an interesting space on its own.
The one-point rule which samples at the point $(\tfrac12,\ldots,\tfrac12)$ has worst-case error equal to zero in this space, as can be seen immediately from~\eqref{eq:wcelin}.
The worst-case error in this multilinear space will show up as part of the worst-case error in the Sobolev space that we will discuss next.
Also note that, by construction, the optimal vertex modified lattice rule has
\begin{align*}
  \wce(Q^{*}; K^{\lin}_{s,\bsgamma})
  &=
  0
  .
\end{align*}
Naturally for $s=1$ also $\wce(T; K^{\lin}_{1,\bsgamma}) = 0$.

\subsection{The Unanchored Sobolev Space of Smoothness~$1$}

The reproducing kernel of the unanchored Sobolev space of smoothness~$1$ is given by
\begin{align*}
  K^{\sob 1}_{s,\bsgamma}(\bsx,\bsy)
  &:=
  \prod_{j=1}^s \left(1 + \gamma_j B_1(x_j) B_1(y_j) + \gamma_j \frac{B_2(\{x_j-y_j\})}{2} \right)
  ,
\end{align*}
and the norm by~\eqref{eq:sobnorm}. (The inner product is built as the tensor product based on the one-dimensional inner product~\eqref{eq:sobip}.)
We note that for functions from the Korobov space with $\alpha=1$
\begin{align*}
  \|f\|_{\sob1,s,\bsgamma} 
  &=
  \|f\|_{\kor1,s,\bsgamma/(2\pi)^{2}}
  \qquad
  \text{for all} \quad f \in \calH(K^{\kor 1}_{s,\bsgamma}),
\end{align*} 
where $\bsgamma/(2\pi)^{2}$ means all weights are rescaled by a factor of $1/(2\pi)^{2}$, which can easily be seen from the one-dimensional case using~\eqref{eq:sobip} and the Fourier series of~$f$, see also \cite{DNP2014}.

Lattice rules were studied in the unanchored Sobolev space in \cite{DNP2014} using the tent-transform and were shown to achieve $O(N^{-1})$ convergence rate without the need for random shifting as was previously known.
A second approach in that paper used full symmetrisation of the point set (reflection around $\tfrac12$ for each combination of dimensions; this is the generalization of the $1$-point rule at $\tfrac12$ for the multilinear space as discussed above, making sure all multilinear functions are integrated exactly).
In a way we can look at vertex modified lattice rules $Q^{\m}$ as being only the symmetrisation of the node $\bszero$ but with different weights.
Using equal weights leads to the rule $T(\cdot; \bsz, N)$ which is the full symmetrisation of the point $\bszero$ (but does not necessarily integrate the multilinear functions exactly).
For the rule $Q^*(\cdot; \bsz, N)$ the weights are chosen in a more intrinsic way such that they integrate multilinear functions exactly and we will concentrate our analysis on this rule.

\section{Error Analysis}

\subsection{Decomposing the Error for the Unanchored Sobolev Space}

We study the worst-case error of using a vertex modified lattice rule in the unanchored Sobolev space.
First note
\begin{multline}\label{eq:Ksob}
  K^{\sob 1}_{s,\bsgamma}(\bsx,\bsy)
  =
  1
  +
  \sum_{\emptyset \ne \setu \subseteq \{1:s\}}
    \prod_{j \in \setu} \gamma_j B_1(x_j) B_1(y_j)
  +
  \sum_{\emptyset \ne \setu \subseteq \{1:s\}}
    \prod_{j \in \setu} \gamma_j \frac{B_2(\{x_j-y_j\})}{2}
  \\
  +
  \sum_{\emptyset \ne \setu \subseteq \{1:s\}}
    \sum_{\emptyset \ne \setv \subset \setu}
      \prod_{j \in \setu} \gamma_j B_1(x_j) B_1(y_j)
      \prod_{j' \in \setv} \gamma_{j'} \frac{B_2(\{x_{j'}-y_{j'}\})}{2}
  .
\end{multline}
From this the following break down of the worst-case error can be obtained.

\begin{proposition}\label{prop:wce}
  The squared worst-case error for a general cubature rule $Q(f) = \sum_{k=1}^{M} w_k \, f(\bsx_k)$, with $\sum_{k=1}^{M} w_{k}=1$, in the unanchored Sobolev space of smoothness~$1$ is given by
  \begin{multline*}
    \wce(Q; K^{\sob 1}_{s,\bsgamma})^{2}
    =
    \wce(Q; K^{\lin}_{s,\bsgamma/12})^2
    +
    \wce(Q; K^{\kor 1}_{s,\bsgamma/(2\pi)^2})^2
    \\+
    \sum_{k,\ell=1}^{M} w_k w_\ell
  \sum_{\emptyset \ne \setu \subseteq \{1:s\}}
    \sum_{\emptyset \ne \setv \subset \setu}
      \prod_{j \in \setu \setminus \setv} \gamma_j B_1(x_{k,j}) B_1(x_{\ell,j})
      \prod_{j' \in \setv} \gamma_{j'} \frac{B_2(\{x_{k,j'}-x_{\ell,j'}\})}{2}
  .
  \end{multline*}
\end{proposition}
\begin{proof}
  This can be found by direct calculation using~\eqref{eq:Ksob} in~\eqref{eq:wce} and comparing terms with the worst-case errors in the Korobov space~\eqref{eq:wceKor} and the multilinear space~\eqref{eq:wcelin}.
\end{proof}

This means our worst-case error is constituted of the worst-case error in the multilinear space (with the weights scaled by $1/12$) and the worst-case error in the Korobov space of smoothness~$1$ (with the weights rescaled by $1/(2\pi)^2$) plus a ``mixture term''.
For the optimal modified lattice rule $Q^{*}$ the error in the multilinear space is zero.
Additionally, the worst-case error in the Korobov space does not change for a vertex modified lattice rule as it just distributes the weight of the point $\bszero$ to the other vertices, but such that the sum of all vertex weights is still $1/N$, see~\eqref{eq:wceQmKor}.

Obviously, in only one dimension, the mixture term is not present as we cannot take both $\setu$ and $\setv$ non-empty, and then the worst-case error in the Sobolev space of smoothness~$1$ equals the worst-case error of the respective lattice rule in the Korobov space of smoothness~$1$ (with rescaled weights) when multilinear functions are integrated exactly.
In two dimensions the mixture term can be rewritten into a nice form as we show in the next proposition which gives the worst-case errors for $s=1$ and~$s=2$.

\begin{proposition}\label{prop:2d-wce}
  For $s=1$ with any $Q(f) = \sum_{k=1}^{M} w_k \, f(x_k)$, where $\sum_{k=1}^{M} w_{k}=1$,
  \begin{align*}
  \wce(Q; K^{\sob 1}_{1,\bsgamma})^2
  &=
  \wce(Q; K^{\lin}_{1,\bsgamma/12})^2
  +
  \wce(Q; K^{\kor 1}_{1,\bsgamma/(2\pi)^2})^2
  .
  \end{align*}
  Specifically the one-dimensional trapezoidal rule, $T(f) = \frac1N \sum_{k=1}^{N-1} f(k/N) + (f(0) + f(1)) / (2N)$, which is equal to the optimal vertex modified rule for $s=1$, gives
  \begin{align*}
  \wce(T; K^{\sob 1}_{1,\bsgamma})
  =
  \wce(Q^*; K^{\sob 1}_{1,\bsgamma})
  &=
  \sqrt{\frac{\gamma_1}{12}} \, \frac1N
  .
  \end{align*}
  For $s=2$ we have for an optimal vertex modified lattice rule $Q^*(\cdot; \bsz, N)$, with $\gcd(z_1,N)=1$ and $\gcd(z_2,N)=1$,
  \begin{align}
  \label{eq:wce2-s2}
  \wce(Q^{*}; K^{\sob 1}_{2,\bsgamma})^2
  &=
  \wce(Q^{*}; K^{\kor 1}_{2,\bsgamma/(2\pi)^2})^2
  +
  \frac{\gamma_1 \gamma_2}{8\, \pi^2 N^2}
  \sum_{j\in\{1,2\}}
  \sum_{\substack{h \ge 1 \\ h \not\equiv 0~(\operatorname{mod}{N})}}
  \frac{\cot^2(\pi h w_j / N)}{h^2}
  ,
  \end{align}
  where we have set $w_1 \equiv z_1^{-1} z_2 \pmod{N}$ and $w_2 \equiv z_2^{-1} z_1 \pmod{N}$, such that $w_2 \equiv w_1^{-1} \pmod{N}$.
  Furthermore
  \begin{align*}
  \wce(Q^{*}; K^{\sob 1}_{2,\bsgamma})^2 
  &>
  \wce(Q^{*}; K^{\kor 1}_{2,\bsgamma/(2\pi)^2})^2
  +
  \frac{\gamma_1 \gamma_2}{8\,\pi^2 N^2}
  \sum_{j\in\{1,2\}}
  \sum_{h = 1}^{N-1}
  \frac{\cot^2(\pi h w_j / N)}{h^2}
  \\
  \wce(Q^{*}; K^{\sob 1}_{2,\bsgamma})^2
  &<
  \wce(Q^{*}; K^{\kor 1}_{2,\bsgamma/(2\pi)^2})^2
  +
  \frac{\gamma_1 \gamma_2}{48 \, N^2}
  \sum_{j\in\{1,2\}}
  \sum_{h = 1}^{N-1}
  \frac{\cot^2(\pi h w_j / N)}{h^2}
  .
  \end{align*}
\end{proposition}
\begin{proof}
For $s=1$ and the trapezoidal rule we see from Proposition~\ref{prop:wce} that we only need to consider the error for the space $K_{1,\gamma_1/(2\pi)^2}^{\kor1}$ since $T = Q^*$ for $s=1$.
So we need to look at the two-fold quadrature of $B_2(\{x-y\})$.
Since this function is periodic the trapezoidal rule $T$ reduces to the standard lattice rule~\eqref{eq:latticerule} such that
\begin{align*}
  \wce(T;K_{1,\bsgamma/(2\pi)^2}^{\kor1})^2
  &=
  \frac{\gamma_1}{N^2} \sum_{k,\ell=0}^{N-1} \frac{B_2((k-\ell \bmod{N})/N)}{2}
  =
  \frac{\gamma_1}{N} \sum_{k=0}^{N-1} \frac{B_2(k/N)}{2}
  =
  \frac{\gamma_1}{12 \, N^2}
  .
\end{align*}

For $s=2$ and
a general cubature rule $Q(f) = \sum_{k=1}^{M} w_{k} \, f(\bsx_{k})$ there are two $2$-dimensional mixture terms in Proposition~\ref{prop:wce}: for $j=1$, $j'=2$ and $j=2$, $j'=1$ we have
{\allowdisplaybreaks
\begin{align}
  \label{eq:double-sum}
  &\hspace{-1mm}\sum_{k,\ell=1}^{M} w_k w_\ell \,
  \gamma_j \, B_1(x_{k,j}) B_1(x_{\ell,j}) \,
  \gamma_{j'} \, \frac{B_2(\{x_{k,j'} - x_{\ell,j'}\})}{2}
  \\
  \nonumber
  &=
  \frac{\gamma_j \, \gamma_{j'}}{(2 \pi)^2}
  \sum_{k=1}^{M} w_k
   B_1(x_{k,j})
  \sum_{\ell=1}^{M} w_{\ell}
  B_1(x_{\ell,j})
  \sum_{0 \ne h \in \Z} \frac{\exp(\twopii h (x_{k,j'} - x_{\ell,j'}))}{h^2}
  \\
  \nonumber
  &=
  \frac{\gamma_j \, \gamma_{j'}}{(2 \pi)^2}
  \sum_{0 \ne h \in \Z} 
  \frac1{h^2}
  \left[
  \sum_{k=1}^{M} w_k
  B_1(x_{k,j})
  \exp(\twopii h x_{k,j'})
  \right]
  \left[ 
  \sum_{\ell=1}^{M} w_{\ell}
  B_1(x_{\ell,j})
  \exp(-\twopii h x_{\ell,j'})
  \right]
  \\
  \nonumber
  &=
  \frac{\gamma_j \, \gamma_{j'}}{(2 \pi)^2}
  \sum_{0 \ne h \in \Z} 
  \frac1{h^2}
  \left|
  \sum_{k=1}^{M} w_k
  B_1(x_{k,j})
  \exp(\twopii h x_{k,j'})
  \right|^2
  ,
\end{align}}%
where we used the Fourier expansion of $B_2$ as given in~\S\ref{sec:Korobov-space}.
We now focus on the 2-dimensional cubature sum inside the modulus.
For the optimal vertex modified lattice rule $Q^{*}$ this cubature sum gives 
\begin{multline*}
  \sum_{k=1}^{M} w_{k} \,
  B_1(x_{k,j}) \,
  \exp(\twopii h x_{k,j'})
  \\=
  \sum_{\bsa \in \{0,1\}^2} w^{*}(\bsa) \, B_1(a_j) \, \exp(\twopii h a_{j'})
  +
  \frac1N
  \sum_{k=1}^{N-1}
  B_1\left(\left\{\frac{z_j k}{N}\right\}\right) \,
  \exp(\twopii h z_{j'} k / N)
  .
\end{multline*}
In the first part the exponential disappears as $\exp(\twopii h a_{j'}) = 1$ for all $\bsa \in \{0,1\}^2$.
Furthermore the whole sum over $\bsa \in \{0,1\}^2$ vanishes as, using $\gcd(z_j, N) = 1$,
\begin{multline*}
  Q^{*}(B_1(x_j); \bsz, N)
  =
  0
  \\=
  \sum_{\bsa \in \{0,1\}^2} w^{*}(\bsa) \, B_1(a_j) + \frac1N \sum_{k=1}^{N-1} B_1\left(\left\{\frac{z_j k}{N}\right\}\right)
  =
  \sum_{\bsa \in \{0,1\}^2} w^{*}(\bsa) \, B_1(a_j)
  ,
\end{multline*}
where the equality to zero follows from the exactness for multilinear functions and the sum over $k$ vanishes due to symmetry.
Thus, using $Q^{*}$ and making use of the forthcoming Lemma~\ref{lem:expsum} and the fact that $\gcd(z_{j'}, N) = 1$, we find, for $w_j = z_j^{-1} z_{j'} \bmod{N}$, with $z_j^{-1}$ the multiplicative inverse of $z_j$ modulo $N$,
\begin{multline*}
  \frac1N
  \sum_{k=1}^{N-1}
  B_1\left(\left\{\frac{z_j k}{N}\right\}\right)
  \exp(\twopii h z_{j'} k / N)
  \\=
  \begin{cases}
    0 & \text{when } h w_j \equiv 0 \pmod{N} , \\
    -\imagunit \cot(\pi h w_j / N) / (2N) & \text{otherwise}.
  \end{cases}
\end{multline*}
It thus follows that, for $Q=Q^*$, each mixture term takes the form
\begin{multline*}
  \frac{\gamma_j \, \gamma_{j'}}{(2 \pi)^2}
  \sum_{0 \ne h \in \Z} 
  \frac1{h^2}
  \left|
  \sum_{k=1}^{M} w_k
  B_1(x_{k,j})
  \exp(\twopii h x_{k,j'})
  \right|^2
  \\=
  \frac{\gamma_j \, \gamma_{j'}}{(4 \pi)^2 N^2}
  \sum_{\substack{0 \ne h \in \Z \\ hw_j \not\equiv 0~(\operatorname{mod}{N})}} 
  \frac{\cot^2(\pi h w_j / N)}{h^2}
  .
\end{multline*}
Making use of $\gcd(w_j, N) = 1$ and using the sign-symmetry on the sum we obtain
{
\allowdisplaybreaks
\begin{align*}
  \frac{2\, \gamma_j \, \gamma_{j'}}{(4 \pi)^2 N^2}
  \sum_{\substack{h \ge 1 \\ hw_j \not\equiv 0~(\operatorname{mod}{N})}}
  \frac{\cot^2(\pi h w_j / N)}{h^2}
  &=
  \frac{\gamma_j \, \gamma_{j'}}{8 \, \pi^2 N^2}
  \sum_{\substack{h \ge 1 \\ h \not\equiv 0~(\operatorname{mod}{N})}}
  \frac{\cot^2(\pi h w_j / N)}{h^2}
  \\
  &=
  \frac{\gamma_j \, \gamma_{j'}}{8\,\pi^2 N^2}
  \sum_{\ell \ge 0}
  \sum_{h = 1}^{N-1}
  \frac{\cot^2(\pi (\ell N + h) w_j / N)}{(\ell N + h)^2}
  \\
  &=
  \frac{\gamma_j \, \gamma_{j'}}{8\,\pi^2 N^2}
  \sum_{\ell \ge 0}
  \sum_{h = 1}^{N-1}
  \frac{\cot^2(\pi h w_j / N)}{(\ell N + h)^2}
  \\
  &=
  \frac{\gamma_j \, \gamma_{j'}}{8\,\pi^2 N^2}
  \sum_{h = 1}^{N-1}
  \frac{\cot^2(\pi h w_j / N)}{h^2}
  \sum_{\ell \ge 0}
  \frac{1}{(\ell N / h + 1)^2}
  \\
  &<
  \frac{\gamma_j \, \gamma_{j'}}{8 \, \pi^2 N^2}
  \sum_{h = 1}^{N-1}
  \frac{\cot^2(\pi h w_j / N)}{h^2} 
  \sum_{\ell \ge 1}
  \frac{1}{\ell^2}
  \\
  &=
  \frac{\gamma_j \, \gamma_{j'}}{48 \, N^2}
  \sum_{h = 1}^{N-1}
  \frac{\cot^2(\pi h w_j / N)}{h^2}
  .
\end{align*}}%
For the upper bound we have set $h=N-1$ in the sum over $\ell \ge 0$ and then used $N/(N-1) > 1$ and $\sum_{\ell\ge1} \ell^{-2} = \pi^2/6$.
The lower bound is easily derived from the same line by considering the case $\ell = 0$ only.
\end{proof}

It is a little bit unfortunate that the $\cot^2$-sum for both $w_1$ and $w_2$ appears in~\eqref{eq:wce2-s2}.
We strongly believe that the infinite sum over $h$ is the same for $w_1$ and $w_2$, and this is equivalent to obtaining the same value for~\eqref{eq:double-sum}.
If this is true than also in the upper and lower bound we just remain with twice either of the sums.
We verified the equality on~\eqref{eq:double-sum} numerically for all $N \le 4001$ and $z \in \{1,\ldots,N-1\}$ with $\gcd(z,N)=1$ and could not find a counter example.
Moreover in Corollary~\ref{cor:wce-Fib}, forthcoming, we show equality to always hold in case of Fibonacci lattice rules.
Therefore we make the following conjecture.


\begin{conjecture}\label{con:symmetry}
  Given integers $z$ and $N$, with $\gcd(z, N) = 1$, we have
  \begin{multline*}
    \sum_{k,\ell=1}^{N-1} B_1(k/N) \, B_2((z (k-\ell) \bmod{N})/N) \, B_1(\ell/N)
    \\=
    \sum_{k,\ell=1}^{N-1} B_1(k/N) \, B_2((z^{-1} (k-\ell) \bmod{N})/N) \, B_1(\ell/N)
    ,
  \end{multline*}
  where $z^{-1}$ is the multiplicative inverse of $z$ modulo~$N$.
\end{conjecture}

The following lemma was used in the proof of Proposition~\ref{prop:2d-wce} for the cubature sum of the linear Bernoulli polynomial in dimension $j$ with a single exponential function in dimension $j'$, taking $\theta = h z_{j'}$.
The lemma is also valid for a product of exponential functions which in the case of lattice rules would give $\theta = \bsh_{\setu} \cdot \bsz_{\setu}$ for some $\setu \subset \{1:s\}$.
\begin{lemma}\label{lem:expsum}
  For $\theta \in \Z$ and $\gcd(z_j, N)=1$, denote by $z_{j}^{-1}$ the multiplicative inverse of $z_{j}$ modulo~$N$, then 
  \begin{align*}
    \frac1N \sum_{k=1}^{N-1} B_1\!\left(\left\{\frac{z_j k}{N}\right\}\right)
          \exp(\twopii \theta \, k / N)
    &=
    \begin{cases}
    0 , & \text{if } \theta \equiv 0 \pmod{N}, \\[1mm]
    \displaystyle \frac{-\imagunit}{2N} \cot(\pi z_j^{-1} \theta / N) , & \text{otherwise}.
    \end{cases}
  \end{align*}
\end{lemma}
\begin{proof}
  With $a = \exp(\twopii z_j^{-1} \theta / N)$ and $z_j^{-1} \theta \not\equiv 0 \pmod{N}$ we have
  \begin{align*}
    \frac1N \sum_{k=1}^{N-1} B_1\!\left(\frac{k}{N}\right) 
      a^k
    &=
    -\frac1{2N} \sum_{k=1}^{N-1} a^k
    +
    \frac1N \sum_{k=1}^{N-1} \frac{k}{N} a^k
    ,
  \end{align*}
  where $\sum_{k=1}^{N-1} a^k = -1$ as $a^N = 1$.
  Now using
  \begin{align*}
    \sum_{k=1}^{N-1} \frac{k}{N} (f(k+1)-f(k))
    &=
    -\frac1N \sum_{k=1}^{N-1} f(k) + \frac{N-1}{N} f(N)
    ,
  \end{align*}
  and, for $a \ne 1$,
  \begin{align*}
    a^k
    &=
    \frac{a^{k+1}}{a-1} - \frac{a^k}{a-1}
  \end{align*}
  we find
  \begin{align*}
    \sum_{k=1}^{N-1} \frac{k}{N} a^k
    &=
    -\frac1N \frac1{a-1} \sum_{k=1}^{N-1} a^k
    +
    \frac{N-1}{N} \frac{a^N}{a-1}
    ,
  \end{align*}
  where again $a^N = 1$ and $\sum_{k=1}^{N-1} a^k = -1$.
  Thus
  \begin{align*}
    \frac1N \sum_{k=1}^{N-1} B_1\!\left(\frac{k}{N}\right) 
      a^k
    &=
    \frac1{2N} + \frac1N \frac1{a-1}
    .
  \end{align*}
  The proof is then completed by taking $t=\pi z_j^{-1} \theta / N$ in the identity $-\mathrm{i} \cot(t) = 1 + 2 / (\exp(2\mathrm{i}\,t) - 1)$.
\end{proof}

\subsection{Upper and Lower Bound}

In Proposition~\ref{prop:2d-wce} we already obtained an upper and a lower bound on $\wce(Q^*;K_{2,\bsgamma}^{\sob1})^2$, but they were in terms of the sum
\begin{align}\label{eq:cot2sum}
  \frac1{N^2}
  \sum_{h=1}^{N-1} \frac{\cot^2(\pi h w/N)}{h^2}
  ,
\end{align}
with $\gcd(w, N) = 1$.
In fact also the sum with $w^{-1}$, the multiplicative inverse of $w$ modulo~$N$, should be considered if Conjecture~\ref{con:symmetry} is false. 
If the conjecture would be false then this can be fixed in the end by assuming $N$ to be large enough (see the remark after Proposition~\ref{prop:wce-2d-logn}).
Note that the sum is $1$-periodic in $t=w/N$ as well as having the symmetry $\cot^2(\pi t) = \cot^2(\pi (1-t)) = \cot^2(-\pi t)$.

Below we will use the series
\begin{align}\label{eq:harmonic}
  H_N(a)
  &:=
  \sum_{h=1}^N \frac1{h^a}
  ,
\end{align}
where we consider $a\ge1$.
This is known as the harmonic number of $N$ of order $a$. If we set $N=\infty$ we get the Riemann zeta function
\begin{align}\label{eq:zeta}
  \zeta(a)
  &:=
  \sum_{h=1}^\infty \frac1{h^a}
  ,
\end{align}
which is finite for $a > 1$.
Since $\zeta(1) = \infty$ we can look at how $H_N(1)$ increases.
For $N\ge3$ we have
\begin{align}\label{eq:harmonic1-bound}
  H_N(1) 
  &\le 
  \frac{11}{6\log(3)} \log(N)
  .
\end{align}
The above elementary bound follows from the definition of the Euler--Mascheroni constant $\lim_{N\to\infty} H_N(1) - \log(N) 
\approx 0.5772$, which converges monotonically from above. 
Solving $H_3(1) = c \log(3)$ results in~\eqref{eq:harmonic1-bound} for $N\ge3$.
We will also make use of the following identity
\begin{align}\label{eq:cot2average}
  \frac1{N-1} \sum_{w=1}^{N-1} \cot^2(\pi w / N) 
  &=
  \frac{N-2}3
  ,
\end{align}
which can be seen by the closed form solution of the Dedekind sum $S(z,N)$ with $z=1$.

The standard approach to show existence of a good generating vector is to prove a good upper bound for the average over all possible generating vectors.
We first show a general lower bound and then an upper bound for the average choice of generating vector on the above $\cot^2$-sum.

\begin{lemma}
  For $N \ge 3$ and any choice of $w$ such that $\gcd(w,N)=1$, the following lower bound holds:
  \begin{align*}
    \frac1{N^2} \sum_{h=1}^{N-1} \frac{\cot^2(\pi h w / N)}{h^2}
    >
    \frac1{6 N^2}
    .
  \end{align*}
\end{lemma}
\begin{proof}
  We have $\{ h w \bmod{N} : h \in \{1,\ldots,N-1\}\} = \{1,\ldots,N-1\}$ since $\gcd(w,N) = 1$, thus
  \begin{align*}
    \sum_{h=1}^{N-1} \frac{\cot^2(\pi h w / N)}{h^2}
    &>
    \sum_{h=1}^{N-1} \frac{\cot^2(\pi h w/ N)}{(N-1)^2}
    =
    \sum_{h=1}^{N-1} \frac{\cot^2(\pi h / N)}{(N-1)^2}
    =
    \frac{N-2}{3(N-1)}
    ,
  \end{align*}
  where we used~\eqref{eq:cot2average}.
\end{proof}

The previous lemma shows that we cannot expect the worst-case error to be better behaving than $1/N$ which is not a surprise as this is the expected convergence for 1D.
We now check what happens if we uniformly pick a $w$ from $\{1,\ldots,N-1\}$ for prime $N\ge3$.
Surprisingly this can be calculated exactly.

\begin{lemma}\label{lem:average-cot2sum}
  For a prime $N \ge 3$, the average over $w \in \{1,\ldots,N-1\}$ of the $\cot^2$-sum~\eqref{eq:cot2sum} is given by
  \begin{align*}
    \frac1{N-1} \sum_{w=1}^{N-1} \frac1{N^2} \sum_{h=1}^{N-1} \frac{\cot^2(\pi h w / N)}{h^2}
    &=
    \frac{N-2}{3N^2} H_{N-1}(2)
    \quad\le\quad
    \frac{\pi^2}{18 N}
    .
  \end{align*}
\end{lemma}
\begin{proof}
  Since $N$ is prime we have $\gcd(h,N)=1$ and thus $\{ h w \bmod{N} : w \in \{1,\ldots,N-1\}\} = \{1,\ldots,N-1\}$.
  Therefore
 \begin{align*}
    \frac1{N-1} \sum_{w=1}^{N-1} \frac1{N^2} \sum_{h=1}^{N-1} \frac{\cot^2(\pi h w / N)}{h^2}
    &=
    \frac1{N^2} \sum_{h=1}^{N-1} \frac1{h^2} \frac1{N-1} \sum_{w = 1}^{N-1} \cot^2(\pi w / N) 
    \\
    &=
    \frac{N-2}{3N^2} \sum_{h=1}^{N-1} \frac1{h^2}
    \le 
    \frac{\zeta(2)}{3N}
    =
    \frac{\pi^2}{18N}
    ,
  \end{align*}
  where we used~\eqref{eq:cot2average}.
\end{proof}

Unfortunately the above result only allows us to say that the expected worst-case error is only as good as the Monte Carlo rate of $N^{-1/2}$ (since the sum~\eqref{eq:cot2sum} appears in the squared worst-case error).
To get a better bound we need another approach.
If we pick the $w$ which gives the best possible value for the \emph{square root} of the sum~\eqref{eq:cot2sum} then this will also be the best value for the sum directly.
Furthermore, using the following inequality, often called ``Jensen's'' inequality, we have
\begin{align}\label{eq:abscot}
  \left( \frac1{N^2} \sum_{h=1}^{N-1} \frac{\cot^2(\pi hw/N)}{h^2} \right)^{1/2}
  &\le
  \frac1{N} \sum_{h=1}^{N-1} \frac{|\cot(\pi hw/N)|}{h}
  .
\end{align}
We now use a popular trick in proving existence:
the value for the best choice $w^*$ to minimize (either side of)~\eqref{eq:abscot} will be at least as small as the average over all possible choices of $w$, thus
\begin{align*}
  \frac1{N} \sum_{h=1}^{N-1} \frac{|\cot(\pi hw^*/N)|}{h}
  &\le
  \frac1{N-1} \sum_{w=1}^{N-1}
  \frac1{N} \sum_{h=1}^{N-1} \frac{|\cot(\pi hw/N)|}{h}
  .
\end{align*}
The next lemma will give an upper bound for the right hand side above.
The argument that the best choice will be at least as good as the average is used numerous times in Ian Sloan's work and is also used inductively in component-by-component algorithms to construct lattice rules achieving nearly the optimal convergence order, see, e.g., \cite{SR2002,SKJ2002,DKS2013}.

\begin{lemma}\label{lem:abscot}
  For a prime $N \ge 3$, the average over $w \in \{1,\ldots,N-1\}$ of the $|\cot|$-sum in~\eqref{eq:abscot} is given by
  \begin{align*}
    \frac1{N-1} \sum_{w=1}^{N-1}
    \frac1{N} \sum_{h=1}^{N-1} \frac{|\cot(\pi hw/N)|}{h}
    &\le
    \frac{H_{N-1}(1)}{N} \frac{6}{\pi} \log(N)
    .
  \end{align*}
\end{lemma}
\begin{proof}
  Similar as in the proof of Lemma~\ref{lem:average-cot2sum} we use the fact that the multiplicative inverse of $h$ exists and we can thus just look at the sum over~$w$.
  For $N\ge3$
  {\allowdisplaybreaks
  \begin{align*}
    \frac1{N-1} \sum_{w=1}^{N-1} |\cot(\pi w / N)|
    &=
    \frac2{N-1} \left[ \cot(\pi/N) + \sum_{w=2}^{(N-1)/2} \cot(\pi w / N) \right]
    \\
    &\le
    \frac2{N-1} \left[ \cot(\pi/N) + \int_{1/N}^{(N-1)/(2N)} \cot(\pi t) \, N \,\rd{t} \right]
    \\
    &=
    \frac2{N-1} \left[ \cot(\pi/N) + \frac{N}{\pi} \, (-\log(2\sin(\pi/(2N)))) \right]
    \\
    &<
    \frac{2.2}{\pi} (1 + \log(4N/3))
    \\
    &<
    \frac3\pi \log(3 N)
    ,
  \end{align*}}%
  where we used $2/\sin(\pi/(2N)) \le 4N/3$ for $N\ge3$, with equality for $N=3$, and some elementary bounds.
\end{proof}

We can now combine the previous results in estimating an upper bound for the worst-case error of a good choice of $w$ for the optimal vertex modified rule $Q^*$ for $s=2$.
From Proposition~\ref{prop:2d-wce}, again using Jensen's inequality by taking square-roots on both sides, we obtain
\begin{align}\label{eq:wce2-sqrt}
  \wce(Q^{*}; K^{\sob 1}_{2,\bsgamma})
  &<
  \wce(Q^{*}; K^{\kor 1}_{2,\bsgamma/(2\pi)^2})
  +
  \frac{\sqrt{\gamma_1 \gamma_2}}{\sqrt{48} \, N}
  \sum_{j\in\{1,2\}}
  \sum_{h = 1}^{N-1}
  \frac{|\cot(\pi h w_j / N)|}{h}
\end{align}
where the sum over $j$ could be replaced by $\sqrt2$ if Conjecture~\ref{con:symmetry} is true.

Piecing everything together we obtain the following result.

\begin{proposition}\label{prop:wce-2d-logn}
  Given a sufficiently large prime $N$, then there exist $w \in \{1,\ldots,N-1\}$ such that the optimal vertex modified rule $Q^*$, with generating vector $\bsz = (1, w)$, has worst-case error in the unanchored Sobolev space for $s=2$ of
  \begin{align*}
    \wce(Q^{*}; K^{\sob 1}_{2,\bsgamma})
    &
    <
    \wce(Q^{*}; K^{\kor 1}_{2,\bsgamma/(2\pi)^2})
    +
    \frac{11 \sqrt{2\,\gamma_1 \gamma_2}}{\pi \sqrt{48} \log3} \frac{\log^2(N)}{N}
    .
  \end{align*}
  If Conjecture~\textup{\ref{con:symmetry}} is true then sufficiently large can be replaced by a prime $N \ge 3$.
\end{proposition}
\begin{proof}  
  From Proposition~\ref{prop:2d-wce} we obtain~\eqref{eq:wce2-sqrt} and combine this with equation~\eqref{eq:harmonic1-bound} and Lemma~\ref{lem:abscot}.
\end{proof}

It is well known that there exist lattice rules for the Korobov space of order~$1$ which have convergence $N^{-1+\delta}$ for $\delta>0$, see, e.g., \cite{SJ94,DKS2013}.
The question of finding a good optimal vertex modified rule for the unanchored Sobolev space now boils down to having $N$ large enough such that the set of good $w$ for the Korobov space and the set of good $w$ for the $|\cot|$-sum overlap.
This is done by showing there exist at least $N/2$ good choices that satisfy twice the average and then necessarily these two sets overlap.
We will not disgress here.
See, e.g., \cite{CKN2010} for such a technique.
Similarly, if the conjecture is not true, then the same technique can be applied by taking $N$ large enough such that all three good sets overlap and one obtains the desired convergence.

\subsection{Fibonacci Lattice Rules}

In \cite{NS94}, Niederreiter and Sloan turn to Fibonacci lattice rules as it is well known they perform best possible in view of many different quality criteria for numerical integration in two dimensions, see, e.g., \cite{Nie92}.
The Fibonacci numbers can be defined recursively by $F_0 = 0$, $F_1 = 1$ and $F_k = F_{k-1} +  F_{k-2}$ for $k \ge 2$.
A Fibonacci lattice rule then takes the number of points a Fibonacci number $N = F_k$ and the generating vector $\bsz = (1, F_{k-1})$, for $k \ge 3$.

We can now show that Conjecture~\ref{con:symmetry} is true for the explicit case of Fibonacci lattice rules.
\begin{lemma}\label{lem:symm-Fib}
  For $z = F_{k-1}$ or $F_{k-2}$ and $N=F_k$, $k\ge3$, we have $\gcd(z, N) = 1$, and
  \begin{multline*}
    \sum_{k,\ell=1}^{N-1} B_1(k/N) \, B_2((z (k-\ell) \bmod{N})/N) \, B_1(\ell/N)
    \\=
    \sum_{k,\ell=1}^{N-1} B_1(k/N) \, B_2((z^{-1} (k-\ell) \bmod{N})/N) \, B_1(\ell/N)
    ,
  \end{multline*}
  where $z^{-1}$ is the multiplicative inverse of $z$ modulo~$N$.
\end{lemma}
\begin{proof}
  It is known that $F_{k-1}^{-1} \equiv \pm F_{k-1} \pmod{F_k}$ with a plus sign for $k$ even and a minus sign for $k$ odd.
  The result follows from the symmetry $B_2(t) = B_2(1-t)$ for $0 \le t \le 1$.
\end{proof}

Combining Lemma~\ref{lem:symm-Fib} with Proposition~\ref{prop:2d-wce} gives then an exact expression for the worst-case error in case of Fibonacci lattice rules.
Note that $N$ does not need to be prime for this proof.
\begin{corollary}\label{cor:wce-Fib}
  For $Q^*_k$ an optimal vertex modified lattice rule based on a Fibonacci lattice rule with generator $(1,F_{k-1})$ modulo~$F_k$, $k\ge4$, we have
  \begin{align*}
    \wce(Q^*_k; K_{2,\bsgamma/(2\pi)^2}^{\sob1})^2
    &=
    \wce(Q^*_k; K_{2,\bsgamma}^{\kor1})^2
    +
    \frac{\gamma_1 \gamma_2}{4\pi^2 N^2}
    \sum_{\substack{h \ge 1 \\ h \not\equiv 0~(\operatorname{mod}{N})}}
  \frac{\cot^2(\pi h w / N)}{h^2}
  \end{align*}
  where $w = F_{k-1}$ and $N=F_k$.
\end{corollary}

\section{Numerics and a Convolution Algorithm}

In this section we restrict ourselves to $N$ prime.
Similar in spirit as \cite{NC2006-prime,Nuy2014} it is possible to evaluate the sum
\begin{align*}
  S_N(z/N)
  &:=
  \frac1{N^2} \sum_{h=1}^{N-1} \frac{\cot^2(\pi h z / N)}{h^2}
\end{align*}
for all $z \in \{1,\ldots,N-1\}$ simultaneously by a (fast) convolution algorithm.
Take a generator for the cyclic group $\Z_N^\times := \{1,\ldots,N-1\} = \langle g \rangle$ and represent $z = \langle g^{\beta} \rangle$ and $h = \langle g^{-\gamma} \rangle$, where $\langle \cdot \rangle$ denotes calculation modulo~$N$.
Then consider for all $0 \le \beta \le N-2$
\begin{align*}
  S_N(\langle g^{\beta} \rangle)
  &=
  \frac1{N^2} \sum_{\gamma=0}^{N-2} \frac{\cot^2(\pi \langle g^{\beta-\gamma} \rangle / N)}{\langle g^{-\gamma} \rangle^2}
  .
\end{align*}
This is the cyclic convolution of two length $N-1$ vectors and can be calculated by an FFT algorithm. 
In Table~\ref{tbl:results} we show the best choice of $z$ obtained by this method and the associated squared worst-case errors.
Instead of using $h^2$ in the denominator of $S_N$ we actually used a generalized zeta function $\zeta(2,h/N)/N^2$ which is the exact value of the infinite sum in Proposition~\ref{prop:2d-wce}.
These results are plotted in Fig.~\ref{fig:wce}. 
Several reference lines have been superimposed with different powers of $\log(N)$.
We note that for this range of $N$ the $\log^2(N)/N$, see Proposition~\ref{prop:wce-2d-logn}, seems to be an overestimate for the square root of the mixing term.
On the other hand, from the figure we see that the total error for this range of $N$ behaves like $\log^{1/2}(N)/N$ for all practical purposes.
It is interesting to compare this behavior with the results in \cite{Ullrich2014,DU2015} which shows a different algorithm for modifying two-dimensional quasi-Monte Carlo point sets to the non-periodic setting (with $M=5N-2$, while here we have $M=N+3$ for 2D) where an upper bound of $\log^{1/2}(N)/N$ is shown (which is also proven to be the lower bound there).


\begin{table}
\centering
{\normalsize \begin {tabular}{ccccc}%
\hline $N$&$z$&$\mathrm{wce} ^2(Q^*,K^{\mathrm{usob} 1}_{2,\bsone })\,=$&$\mathrm{wce} ^2(Q^*,K^{\mathrm{kor} 1}_{2,\bsone /(2\pi )^2})\,\,+$&mixing term\\\hline %
\ensuremath {17}&\ensuremath {5}&\ensuremath {2.16\cdot 10^{-3}}&\ensuremath {1.92\cdot 10^{-3}}&\ensuremath {2.39\cdot 10^{-4}}\\%
\ensuremath {37}&\ensuremath {11}&\ensuremath {5.33\cdot 10^{-4}}&\ensuremath {4.57\cdot 10^{-4}}&\ensuremath {7.63\cdot 10^{-5}}\\%
\ensuremath {67}&\ensuremath {18}&\ensuremath {1.73\cdot 10^{-4}}&\ensuremath {1.46\cdot 10^{-4}}&\ensuremath {2.66\cdot 10^{-5}}\\%
\ensuremath {131}&\ensuremath {76}&\ensuremath {4.67\cdot 10^{-5}}&\ensuremath {3.92\cdot 10^{-5}}&\ensuremath {7.47\cdot 10^{-6}}\\%
\ensuremath {257}&\ensuremath {76}&\ensuremath {1.37\cdot 10^{-5}}&\ensuremath {1.12\cdot 10^{-5}}&\ensuremath {2.47\cdot 10^{-6}}\\%
\ensuremath {521}&\ensuremath {377}&\ensuremath {3.48\cdot 10^{-6}}&\ensuremath {2.83\cdot 10^{-6}}&\ensuremath {6.48\cdot 10^{-7}}\\%
\ensuremath {1{,}031}&\ensuremath {743}&\ensuremath {9.75\cdot 10^{-7}}&\ensuremath {7.81\cdot 10^{-7}}&\ensuremath {1.94\cdot 10^{-7}}\\%
\ensuremath {2{,}053}&\ensuremath {794}&\ensuremath {2.70\cdot 10^{-7}}&\ensuremath {2.13\cdot 10^{-7}}&\ensuremath {5.70\cdot 10^{-8}}\\%
\ensuremath {4{,}099}&\ensuremath {2{,}511}&\ensuremath {7.06\cdot 10^{-8}}&\ensuremath {5.53\cdot 10^{-8}}&\ensuremath {1.53\cdot 10^{-8}}\\%
\ensuremath {8{,}209}&\ensuremath {3{,}392}&\ensuremath {1.88\cdot 10^{-8}}&\ensuremath {1.46\cdot 10^{-8}}&\ensuremath {4.19\cdot 10^{-9}}\\%
\ensuremath {16{,}411}&\ensuremath {6{,}031}&\ensuremath {4.82\cdot 10^{-9}}&\ensuremath {3.73\cdot 10^{-9}}&\ensuremath {1.09\cdot 10^{-9}}\\%
\ensuremath {32{,}771}&\ensuremath {20{,}324}&\ensuremath {1.26\cdot 10^{-9}}&\ensuremath {9.71\cdot 10^{-10}}&\ensuremath {2.91\cdot 10^{-10}}\\%
\ensuremath {65{,}537}&\ensuremath {25{,}016}&\ensuremath {3.34\cdot 10^{-10}}&\ensuremath {2.55\cdot 10^{-10}}&\ensuremath {7.90\cdot 10^{-11}}\\%
\ensuremath {131{,}101}&\ensuremath {80{,}386}&\ensuremath {8.97\cdot 10^{-11}}&\ensuremath {6.79\cdot 10^{-11}}&\ensuremath {2.18\cdot 10^{-11}}\\%
\ensuremath {262{,}147}&\ensuremath {159{,}921}&\ensuremath {2.30\cdot 10^{-11}}&\ensuremath {1.74\cdot 10^{-11}}&\ensuremath {5.64\cdot 10^{-12}}\\\hline %
\end {tabular}%
}
\caption{Optimal choices of generating vector $(1, z)$ for a selection of prime $N$ for the unanchored Sobolev space of order~$1$}\label{tbl:results}
\end{table}

\begin{figure}
  \centering
  \includegraphics{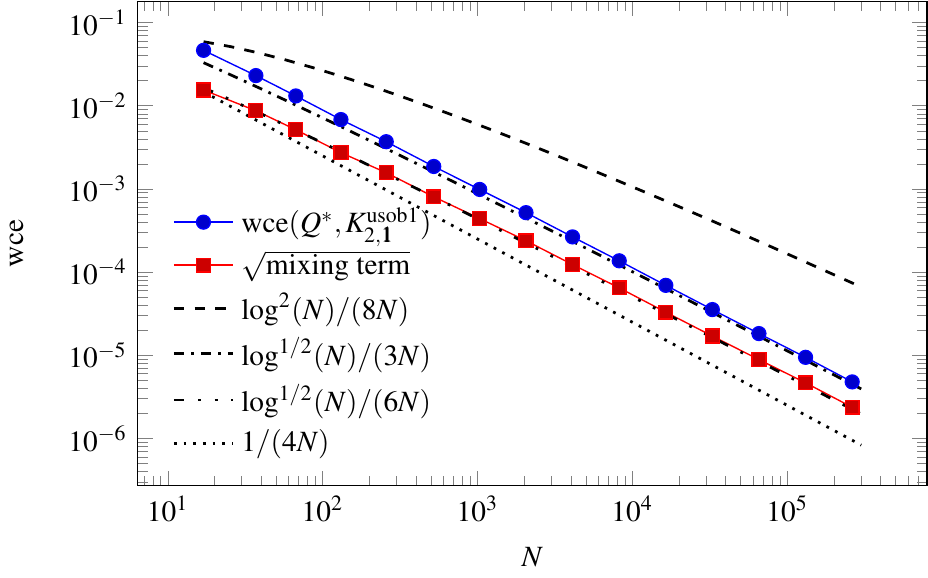}
  \caption{Plot of optimal worst-case error and square root of mixture term from Table~\ref{tbl:results}}\label{fig:wce}
\end{figure}

\section{Conclusion}

In this paper we revisited (optimal) vertex modified lattice rules \cite{NS93,NS94,NS96} introduced by Niederreiter  and Sloan, and studied their error in the unanchored Sobolev space which is one of the typical reproducing kernel Hilbert spaces used to study lattice rules nowadays.
The analysis makes use of a breakdown of the squared worst-case error into the squared worst-case error in a multilinear space, the Korobov space and an additional ``mixture'' term where combinations of basis functions from those two previous spaces appear.
For $s=2$ we showed that there exist optimal vertex modified lattice rules for which the square root of the mixture term converges like $N^{-1} \log^2(N)$.
Because of the $2^s$ cost of evaluating the integrand on all vertices of the unit cube, it does not look very interesting to extend the analysis to an arbitrary number of dimensions.
Although we restricted our detailed analysis to the case $s=2$, we remark that a similar breakdown was achieved in terms of the $L_2$ discrepancy in \cite{RJ2000}, which shows that the cost of $2^s$ vertices still pays off for $s<12$ in their numerical tests.
Such tests would also be useful for the analysis in this paper, as would a component-by-component algorithm for $s>2$.
Finally, a comparison with the bounds in \cite{Ullrich2014,DU2015} suggests the power of the $\log(N)$ term could be improved, as is hinted at by our numerical results.
These are suggestions for future work.

\begin{acknowledgement}
We thank Jens Oettershagen for useful comments and pointers to \cite{Ullrich2014,DU2015}.
We also thank the Taiwanese National Center for Theoretical Sciences (NCTS) - Mathematics
    Division, and the National Taiwan University (NTU) - Department of Mathematics,
    where part of this work was carried out.
We thank the referees for their helpful comments and
acknowledge financial support from the KU~Leuven research
fund (OT:3E130287 and C3:3E150478).
\end{acknowledgement}

%

\bibliographystyle{spmpsci} 
\bibliography{mrabbrev,bib-abbrev}

%

\end{document}